\documentclass[a4paper, 12pt]{article}
\usepackage{amsmath}
\usepackage{amscd}
\usepackage{amsthm}
\usepackage{amssymb}
\usepackage{verbatim}
\usepackage{mathrsfs}

\newtheorem{thm}{Theorem}[section]
\newtheorem{prop}[thm]{Proposition}
\newtheorem{cor}[thm]{Corollary}
\newtheorem{conj}[thm]{Conjecture}

\newcommand{\bbZ}{\mathbb{Z}}
\newcommand{\tate}{\hat{H}}
\newcommand{\cHom}{\mathscr Hom}

\DeclareMathOperator{\Hom}{Hom} 

\DeclareMathOperator{\rk}{rk}
\DeclareMathOperator{\hrk}{hrk}
\DeclareMathOperator{\image}{im}
\DeclareMathOperator{\Ext}{Ext}

\def\maprt#1{\smash{\,\mathop{\longrightarrow}\limits^{#1}\,}}

\begin{document}
\title{Free Actions on Products of Spheres at High Dimensions}
\author{Osman Berat Okutan and Erg\" un Yal\c c\i n\thanks{Partially supported by T\" UB\. ITAK-TBAG/110T712.}}
\maketitle

\begin{abstract} A classical conjecture in transformation group theory states that if $G=(\bbZ/p)^r$ acts freely on a product of $k$ spheres $S^{n_1} \times \cdots \times S^{n_k}$, then $r\leq k$. We prove this conjecture in the case where the average of the dimensions $\{n_i\}$ is large compared to the differences $|n_i-n_j|$ between the dimensions. 
\end{abstract}

\section{Introduction}

Let $G$ be a finite group. The rank of $G$, denoted by $\rk (G)$, is defined as the largest integer $s$ such that $(\bbZ/p)^s\leq G$ for some prime $p$. It is known that $G$ acts freely and cellularly on a finite complex homotopy equivalent to a sphere $S^n$ if and only if $\rk (G)=1$. This follows from the results due to P.A. Smith \cite{smith} and R. Swan \cite{swan}. As a generalization of this, it has been conjectured by Benson-Carlson \cite{benson-carlson} that $\rk (G)=\hrk(G)$ where $\hrk(G)$ is defined as the smallest integer $k$ such that $G$ acts freely and cellularly on a finite CW-complex homotopy equivalent to  a product of $k$ spheres. This conjecture is often referred to as the rank conjecture. Note that one direction of the Benson-Carlson conjecture is the following statement:  

\begin{conj}\label{conj:main} Let $p$ be a prime. If $G=(\bbZ/p)^r$ acts freely and cellularly on a finite CW-complex $X$ homotopy equivalent to $S^{n_1} \times \dots \times S^{n_k}$, then $r\leq k$.
\end{conj}

This conjecture is a classical conjecture which has been studied intensely through 80's and it has been proven that the conjecture is true under some additional assumptions. For example it is known that when the dimensions of the spheres are all equal, i.e., $n=n_1=\cdots =n_k$, then the conjecture is true for all primes $p$ and for all positive integers $n$ except when $p=2$ and $n=3,7$. This was proved by G. Carlsson \cite{carlsson} in the case where the $G$-action on the homology of $X$ is trivial and the general case is due to Adem-Browder \cite{adem-browder}. The $p=2$ and $n=1$ case was proven later in \cite{yalcin}. More recently, B. Hanke \cite{hanke} proved that Conjecture \ref{conj:main} is true in the case where $p\geq 3 \dim X$, i.e., when the prime $p$ is large compared to the dimension of the space. In this paper, we prove Conjecture \ref{conj:main} for the other extreme, i.e., when the average of the dimensions is high compared to the differences between them.
 
\begin{thm}\label{thm:maintheorem} Let $G=(\bbZ/p)^r$ and $k,l$ be positive integers. Then there exists an integer $N$ (depending on $k,l$ and the group $G$) such that if $G$ acts freely and cellularly on a finite dimensional CW-complex $X$ homotopy equivalent to $S^{n_1} \times \dots \times S^{n_k}$ where $n_i \geq N$ and $|n_i-n_j|\leq l$ for all $i,j$, then $r\leq k$.
\end{thm}

The proof follows from a theorem of Browder \cite{browder} which gives a restriction on the order of groups acting freely on a finite dimensional CW-complex in terms of homology groups of the complex.  We also use a method of gluing homology groups at different dimensions which we first saw in a paper by Habegger \cite{habegger} and a crucial result on the exponents of cohomology groups of elementary abelian $p$-groups which is due to Pakianathan \cite{pakianathan}.

At the end of the paper we also prove a generalization of Theorem \ref{thm:maintheorem} to non-free actions which was suggested to us by A. Adem.

The paper is organized as follows: In Section \ref{sect:hypercohomology}, we list some well-known results about hypercohomology and in Section \ref{sect:Habegger}, we introduce Habegger's theorem on gluing homology at different dimensions. In Section \ref{sect:exponents}, we discuss the exponents of Tate cohomology groups and in Section \ref{sect:maintheorem}, we prove Theorem \ref{thm:maintheorem} which is our main theorem.

\section{Tate Hypercohomology}\label{sect:hypercohomology}

Let $G$ be a finite group and $M$ be a $\bbZ G$-module. The Tate cohomology of $G$ with coefficients in $M$ is defined as follows
$$ \tate^i(G,M):=H^i(\Hom_G(F_*,M)) $$
for all $i \in \bbZ$, where $F_*$ is a complete $\bbZ G$-resolution of $\bbZ$ (see \cite[p.~134]{brown}). We can generalize this and define Tate hypercohomology of $G$ with coefficients in a chain complex $C_*$ of $\bbZ G$-modules. To do this, we need to extend the contravariant functor $\Hom_G(-,M)$ to $\cHom_G(-,C_*)$. We will define it as in Brown (see \cite[p.~5]{brown}), but instead of defining it as a chain complex, we consider it as a cochain complex.

For all $n \in \bbZ$, let $\cHom_G(C_*,D_*)^n$ denote the set of graded $G$-module homomorphisms of degree $-n$ and define the boundary map $\delta^n$ by $\delta^n(f)= f\partial - (-1)^n \partial f$. Note that $\cHom_G(-,C_*)$ (resp. $\cHom_G(C_*,-)$) becomes a covariant (resp. contravariant) functor from the category of chain complexes of $\bbZ G$-modules to the category of cochain complexes of abelian groups. Also, if $C_*$ is a chain complex concentrated at 0 with $C_0=M$, then $\cHom_G(-,C_*)$ is naturally equivalent to the functor $\Hom_G(-,M)$.

Now, we define the Tate hypercohomology of a finite group $G$ with coefficients in $C_*$ as follows:
\[\tate^i(G,C_*):=H^i(\cHom_G(F_*,C_*))\]
for all $i\in \bbZ$, where $F_*$ is a complete $\bbZ G$-resolution of $\bbZ$. We immediately have $\tate^i(G,\Sigma C_*)\cong \tate^{i+1}(G,C_*)$, where $(\Sigma C_*)_i=C_{i-1}$ for all $i$. Therefore, if $C_*$ is a chain complex concentrated at $n$, then $\tate^i(G, C_*)\cong \tate^{i+n}(G,C_n)$. Also note that given a short exact sequence of chain complexes
\[0 \to C_* \to D_* \to E_* \to 0 \]
of $\bbZ G$-modules,
there is a long exact sequence of the following form
\[\cdots \to \tate^i(G,C_*) \to \tate^i(G,D_*) \to \tate^i(G,E_*) \to \tate^{i+1}(G,C_*) \to \cdots .\]

An important property of $\cHom$ functor is that if $P_*$ is a chain complex of projective $\bbZ G$-modules and $f_*:C_* \to D_*$ a weak equivalence of nonnegative chain complexes of $\bbZ G$-modules, then $f_*:\cHom_G(P_*,C_*) \to \cHom_G(P_*,D_*)$ is also a weak equivalence (see \cite[p.~29]{brown}). Actually, Brown proves this result by assuming $P_*$ is nonnegative and $C_*$ and $D_*$ are arbitrary, but the same proof remains true if we assume $P_*$ is arbitrary and $C_*$ and $D_*$ are nonnegative.  Using this, we obtain the following proposition:

\begin{prop}\label{concentrated homology lemma}
If $C_*$ is a nonnegative chain complex of $\bbZ G$-modules with homology concentrated at dimension $n$ and $H_n(C_*)=M$, then $\tate^i(G,C_*)\cong \tate^{i+n}(G,M)$.
\end{prop}
\begin{proof}
Let $Z_n$ denote the group of $n$-cycles in $C_*$. We have the following weak equivalences:
\[\minCDarrowwidth20pt\begin{CD}
D_* :  \cdots @>>> C_{n+1}   @>>> Z_n  @>>> 0   @>>> \cdots \\
       @.   @VidVV         @VVV      @VVV     @.       \\
C_* : \cdots @>>> C_{n+1}   @>>> C_n  @>>> C_{n-1}  @>>> \cdots
\end{CD}\]
and
\[\minCDarrowwidth20pt\begin{CD}
D_* : \cdots  @>>> C_{n+1}   @>>> Z_n  @>>> 0  @>>> \cdots \\
     @.   @VVV           @VVV      @VVV     @.       \\
E_* : \cdots @>>> 0         @>>> M    @>>> 0   @>>> \cdots.
\end{CD}\]
Therefore, $\tate^i(G,C_*) \cong \tate^i(G,D_*) \cong \tate^i(G,E_*) \cong \tate^{i+n}(G,M)$.
\end{proof}

An exact sequence $K \maprt{f} L \maprt{g} M$ of $\bbZ G$-modules is called {\it admissible} if the inclusion map $\image (g) \hookrightarrow M$ is $\bbZ$-split (see \cite[p.~129]{brown}). A $\bbZ G$-module $M$ is called relatively injective if $\Hom_G(-,M)$ takes an admissible exact sequence to an exact sequence of abelian groups. Projective $\bbZ G$-modules are relatively injective (see \cite[p.~130]{brown}). Since a complete $\bbZ G$-resolution $F_*$ of $\bbZ$ is an exact sequence of free $\bbZ G$-modules, the sequence $F_{i+1} \to F_i \to F_{i-1}$ is admissible for all $i$. Hence if $P$ is a  projective $\bbZ G$-module, then the Tate cohomology group $\tate^i(G,P)=0$ for all $i$. This result generalizes to hypercohomology. 

\begin{prop}\label{trivial tate cohomology lemma}
If $P_*$ is a chain complex of projective $\bbZ G$-modules which has finite length, then $\tate^i(G,P_*)=0$ for all $i$.
\end{prop}

\begin{proof}  Recall that we say a chain complex $C_*$ has finite length if there are integers $n$ and $m$ such that $C_i=0$ for all $i>n$ and $i<m$. By shifting $P_*$ if necessary, we can assume that $P_*$ is a finite dimensional nonnegative chain complex and prove the proposition by an easy induction on the dimension of $P_*$. 
\end{proof}

We say that two chain complexes $C_*$ and $D_*$ are {\it freely equivalent} if there is a sequence of chain complexes $C_*=E^0_*, \dots, E^n_*=D_*$ such that either $E_*^i$ is an extension of $E_* ^{i-1}$ or $E_* ^{i-1}$ is an extension of $E_*^i$ by a finite length chain complex of free modules. Note that we say a chain complex $D_*$ is an extension of $C_*$ by a finite length chain complex of free modules if there is short exact sequence of chain complexes either of the form $0 \to C_* \to D_* \to F_* \to 0$ or of the form $0 \to F_* \to D_* \to C_* \to 0$, where $F_*$ is a finite length chain complex of free modules. As a corollary of Proposition \ref{trivial tate cohomology lemma}, we have:

\begin{cor}\label{free equivalence lemma}
If two chain complexes $C_*$ and $D_*$ are freely equivalent, then $\tate^i(G,C_*) \cong \tate^i(G,D_*)$ for all $i$.
\end{cor}

Before we conclude this section, we would like to note that there is a hypercohomology spectral sequence which 
converges to the Tate hypercohomology $\tate ^* (G, C_*)$ for a given chain complex $C_*$ of $\bbZ G$-modules. One way to obtain this spectral sequence is to consider the double complex $D^{p,q}=\Hom _{G} (F_p, C_{-q})$ where the vertical and horizontal differentials are given by $\delta _0=\Hom ( - , \partial )$ and $\delta _1=\Hom (\partial , -)$. Note that the total complex ${\rm Tot } D^{*,*}$ with $${\rm Tot}^n D^{*,*}=\bigoplus _{p+q=n} D^{p,q}$$ and $\delta ^n=\delta _0 - (-1)^n \delta _1$ is a cochain complex homotopy equivalent to the cochain complex $\cHom _G (F_* , C_*)$. Filtering this double complex with respect to the index $p$ and then with respect to the index $q$, we obtain two spectral sequences $${}^I E_2 ^{p,q} =\tate ^p (G, H_{-q} (C_*))\Rightarrow  \tate ^{p+q} (G, C_*)$$
$${}^{II} E_1 ^{p,q} =\tate ^q (G, C_{-p})\Rightarrow  \tate ^{p+q} (G, C_*).$$
Note that using these two spectral sequences it is possible to give alternative proofs for Proposition \ref{concentrated homology lemma} and \ref{trivial tate cohomology lemma}.

\section{Habegger's Theorem}
\label{sect:Habegger}

In \cite[p.~433-434]{habegger}, Habegger uses a technique to ``glue" homology groups of a chain complex at different dimensions. This technique will be crucial in the proof of Theorem \ref{thm:maintheorem}, so we give a proof for it here. Before we state Habegger's theorem, we recall the definition of syzygies of modules. 

For every positive integer $n$, the $n$-th syzygy of a $\bbZ G$-module $M$ is defined as the kernel of $\partial_{n-1}$ in a partial resolution of the form 
$$ P_{n-1}\maprt{\partial _{n-1}} \cdots \to P_1 \maprt{\partial _1} P_0 \to M \to 0$$
where $P_0,\dots, P_{n-1}$ are projective $\bbZ G$-modules. We denote the $n$-th syzygy of $M$ by $\Omega ^n M$ and by convention we take $\Omega ^0M=M$.

The $n$-th syzygy of a module $M$ is well-defined only up to stable equivalence.  Recall that two $\bbZ G$-modules $M$ and $N$ are called stably equivalent if there are projective $\bbZ G$-modules $P$ and $Q$ such that $M\oplus P \cong N \oplus Q$. Well-definedness of syzygies up to stable equivalence follows from a generalization of Schanuel's lemma (see \cite[p.~193]{brown}). Since for any two stably equivalent modules $M$ and $N$, we have $\tate ^i (G, M)\cong \tate ^i (G, N)$ for all $i$, we will ignore the fact that syzygies are well-defined only up to stable equivalence and treat $\Omega ^n M$ as a unique module depending only on $M$ and $n$. Alternatively, one can fix a resolution for every $\bbZ G$-module $M$ and define $\Omega ^n M$ as the kernel of $\partial _{n-1}$ in this unique resolution. 

\begin{thm}[Habegger \cite{habegger}]\label{gluing lemma}
Let $C_*$ be a chain complex of $\bbZ G$-modules and $n,m$ are integers such that $m<n$. If $H_k(C_*)=0$ for all $k$ with $m<k<n$, then $C_*$ is freely equivalent to a chain complex $D_*$ such that

(i) $H_i(D_*)=H_i(C_*)$ for every $i \neq n,m$;

(ii) $H_m(D_*)=0$, and;

(iii) there is an exact sequence of $\bbZ G$-modules
\[0 \to H_n(C_*) \to H_n(D_*) \to \Omega^{n-m} H_m(C_*) \to 0.\]
\end{thm}

\begin{proof}
Let $F_{n-1} \to \cdots \to F_m \to H_m(C_*) \to 0$ be an exact sequence where $F_i$'s are free $\bbZ G$-modules.  Consider the following diagram
\[\minCDarrowwidth20pt\begin{CD}
\cdots @>>> 0   @>>> F_{n-1} @>>> \cdots @>>> F_m @>>> H_m(C_*) @>>> 0    @>>> \cdots \\
    @.       @.           @.       @.       @.   @VidVV        @VVV         \\
\cdots @>>> C_n @>>> C_{n-1} @>>> \cdots @>>> Z_m @>>> H_m(C_*) @>>> 0    @>>> \cdots
\end{CD}\]
where $Z_m$ denotes the group of $m$-cycles in $C_*$. Since all $F_i$'s are projective and the bottom row has no homology below dimension $n$, the identity map extends to a chain map between rows. Notice that this chain map gives a chain map $f_*:F_* \to C_*$ as follows
\[\minCDarrowwidth20pt\begin{CD}
\cdots @>>> 0 @>>> F_{n-1} @>>> \cdots @>>> F_m       @>>> 0        @>>> \cdots \\
@.   @VVV     @Vf_{n-1}VV   @.   @Vf_mVV    @VVV  \\
\cdots @>>> C_n @>>> C_{n-1}  @>>> \cdots @>>> C_m  @>>> C_{m-1}  @>>>\cdots .
\end{CD}\]
where the maps $f_i : F_i \to C_i$ for $i>m$ are the same as the maps in the first diagram above. The map $f_m: F_m \to C_m$ is defined as the composition $$F_m \maprt{f_m '} Z_m \hookrightarrow C_m$$ where $f_m': F_m \to Z_m$ is the map defined as the lifting of the identity map in the first diagram.

Now, let $D_*$ be the mapping cone of $f_*$. We have the following short exact sequence of the form
\[ 0 \to C_* \to D_* \to \Sigma F_* \to 0,\]
so $C_*$ is freely equivalent to $D_*$. The corresponding long exact sequence of homology groups is
\[\minCDarrowwidth20pt\begin{CD} \cdots @>>> H_i(F_*) @>f_*>> H_i(C_*) @>>> H_i(D_*) @>>> H_{i-1}(F_*) @>>> \cdots .\end{CD}\]

Assume first that $n>m+1$. Then $F_*$ has at least two terms and its homology is nonzero only at two dimensions $n-1$ and $m$. So, $H_i(C_*)\cong H_i(D_*)$ for all $i$ such that $i\neq m, m+1, n-1, n$. At dimension $m$, the map $f_*:H_m(F_*) \to H_m(C_*)$ is an isomorphism, so we get $H_m(D_*)=H_{m+1}(D_*)=0$. At dimension $n-1$, we have $H_{n-1} (C_*)=0$, so we get $H_{n-1}(D_*)=0$. We also have a short exact sequence of the form
\[\minCDarrowwidth20pt\begin{CD} 0 @>>> H_n (C_*) @>>> H_n (D_*) @>>> H_{n-1}(F_*) @>>>  0. \end{CD}\]
Since $H_{n-1} (F_*)\cong \Omega ^{n-m} (H_m(C_*))$, this gives the desired result.

If $n=m+1$, then $F_*$ has a single term $F_m$, so we have a sequence of the form
\[\minCDarrowwidth20pt\begin{CD} 0 @>>> H_{n} (C_*) @>>> H_n (D_*) @>>> F_m @>f_*>>  H_m(C_*) @>>> H_m (D_*) @>>> 0. \end{CD}\]
Since $f_*$ is surjective by construction, we conclude that $H_m(D_*)=0$ and there is a short exact sequence of the form 
\[\minCDarrowwidth20pt\begin{CD} 0 @>>> H_n (C_*) @>>> H_n (D_*) @>>> \Omega ^1 ( H_m(C_*)) @>>>  0 \end{CD}\]
as desired.
\end{proof}

\section{Exponents of Tate Cohomology Groups}
\label{sect:exponents}

To prove the main theorem, we need some results about the exponents of Tate cohomology groups. We first recall some 
definitions. The exponent of a finite abelian group $A$ is defined as the smallest positive integer $n$ such that $na=0$ for all $a\in A$. We denote the exponent of $A$ by $exp A$. Note that if $A\to B \to C $ is an exact sequence of finite abelian groups, then $exp B$ divides $exp A \cdot exp C $. In this situation we sometimes write $exp B/exp A $ divides $exp C$ to refer to the same fact even though $exp B / exp A$ may not be an integer in general.   

The first result we prove is a proposition on the exponent of Tate cohomology group with coefficients in a filtered module. First let us explain the terminology that we will be using throughout the paper. Let $M$ be a $\bbZ G$-module and $A_1, A_2,\dots, A_n$ be a sequence of $\bbZ G$-modules. If $M$ has a filtration $0 = M_0 \subseteq M_1 \subseteq \dots \subseteq M_n=M$ such that  $M_j/M_{j-1}\cong A_j$ for all $j$, then we say $M$ has a filtration with sections $A_1-A_2-\cdots -A_n$.

\begin{prop}\label{filtration lemma} Let $M$ be a $\bbZ G$-module which has a filtration with sections $A_1 - A_2-\dots - A_n$. Then, $exp \tate^i(G,M)$ divides $\prod_{j=1}^n exp \tate^i(G,A_j)$.
\end{prop}
\begin{proof}
Let $0=M_0 \subseteq M_1 \subseteq \dots  \subseteq M_n=M$ be the filtration of $M$ with the sections as above. Then for every $j$, we have an exact sequence of $\bbZ G$-modules
\[0 \to M_{j-1} \to M_j \to A_j \to 0,\]
which gives a long exact Tate cohomology sequence of the following form
\[\dots \to \tate^i(G,M_{j-1}) \to \tate^i(G,M_j) \to \tate^i(G,A_j) \to \dotsb . \]
From this we observe that $exp \tate^i(G,M_j)/exp \tate^i(G,M_{j-1})$ divides the exponent of $\tate^i(G,A_j)$. Multiplying these relations through all $j=1,\dots,n$, we get $exp \tate^i(G,M)$ divides $\prod_{j=1}^n exp \tate^i(G,A_j)$.
\end{proof}

In \cite{browder}, Browder proves a theorem which gives an upper bound on the order of a finite group $G$ in terms of the exponents of cohomology groups with coefficients in homology groups of a CW-complex on which $G$ acts freely. Since we use this theorem in the proof of our main theorem, we state it below and give a proof for it. The proof we give here is slightly different than the original proof. It uses Theorem \ref{gluing lemma} and  Proposition \ref{filtration lemma}.

\begin{thm}[Browder \cite{browder}]\label{thm:browder}
Let $C_*$ be a nonnegative, free, connected chain complex of dimension $n$. Then $|G|$ divides $\prod_{j=1}^n expH^{j+1}(G,H_j(C_*))$.
\end{thm}

\begin{proof} Let us take $C^{(0)}_*=C_*$ and for $j=1$ to $n$, define $C^{(j)}_*$ to be the chain complex obtained by $C^{(j-1)}_*$ by applying the method in Theorem \ref{gluing lemma} for the dimensions $n-j$ and $n$. Since $C_*$ is a finite dimensional chain complex of free $\bbZ G$-modules, by Proposition \ref{trivial tate cohomology lemma}, $\tate^i(G,C_*)=0$ for all $i$. Hence by Corollary \ref{free equivalence lemma} and Theorem \ref{gluing lemma}, we have $\tate^i(G,C^{(j)}_*)=0$ for all $i,j$. Notice that $C^{(n)}_*$ is a chain complex with homology concentrated at $n$. Let us denote the homology of $C^{(n)}_*$ at $n$ by $M$. Hence, by Proposition \ref{concentrated homology lemma}, we have $\tate^i(G,M)=0$ for all $i$. By Theorem \ref{gluing lemma}, $M$ has a filtration $$0 \subseteq H_n(C^{(0)}_*) \subseteq \cdots \subseteq H_n(C^{(n-1)}_*) \subseteq H_n(C^{(n)}_*)=M$$ with sections $H_n(C_*)-\Omega^1H_{n-1}(C_*)-\dots -\Omega^{n-1}H_1(C_*)-\Omega^nH_0(C_*)$. If we let $N:= H_n(C^{(n-1)}_*)$, then $N$ has a filtration with sections $H_n(C_*)-\Omega^1H_{n-1}(C_*)-\dots -\Omega^{n-1}H_1(C_*)$ and there is a short exact sequence of the form
\[0 \to N \to M \to \Omega^n H_0 (C_*) \to 0.\]
Note that $H_0(C_*)\cong \bbZ$, so we obtain an exact sequence of the form
\[\cdots \to \tate^n(G,M) \to \tate^n(G,\Omega^n\bbZ) \to \tate^{n+1}(G,N) \to \tate^{n+1}(G,M)\to \cdots .\]

Since $\tate ^i (G, M)=0$ for all $i$, we obtain $\tate^{n+1}(G,N) \cong \tate^n(G,\Omega^n\bbZ) \cong \tate^0(G,\bbZ)\cong \bbZ/|G|$. Hence by Proposition \ref{filtration lemma}, we get $|G|=exp \tate^{n+1}(G,N)$ divides the product 
$$\prod _{j=1} ^n exp \tate^{n+1}(G,\Omega^{n-j}H_j(C_*)).$$ Since $\tate ^{n+1} (G, \Omega ^{n-j} H_j (C_* ))\cong H^{j+1}(G,H_j(C_*))$, this gives the desired result.
\end{proof}

As a corollary of Theorem \ref{thm:browder}, Browder gives a proof for a theorem of G. Carlsson \cite{carlsson} which says that if $G=(\bbZ /p)^r$ acts freely on a finite dimensional CW-complex $X\simeq(S^n )^k$ with trivial action on homology, then $r\leq k$. The main observation is that when  $G=(\bbZ/p)^r$ and $M$ is a trivial 
$\bbZ G$-module, the exponent of $H^i(G,M)$ divides $p$ for all $i \geq 1$. This follows easily by induction on $r$ using properties of the transfer map in group cohomology. So, from the relation given in Theorem \ref{thm:browder}, one obtains that if $G$ acts freely on a finite dimensional CW-complex $X\simeq (S^n )^k$ with trivial action on homology, then $|G|=p^r$ divides $p^k$, which gives $r\leq k$.

Note that the assumption that $G$ acts trivially on the homology of $X$ is crucial in the above argument since for an arbitrary $\bbZ G$-module, the exponent of $H^i (G, M)$ can be as large as the order of $|G|$. In fact, if we take $M=\Omega ^i (\bbZ)$ for some positive integer $i$, then we have $H^i (G, M)\cong \bbZ /|G|$, so the exponent of $H^i (G,M)$ is equal to $|G|$ in this case. Taking the direct sum of all such modules over all $i$, one can obtain a $\bbZ G$-module $M$ such that the exponent of $H ^i (G, M)$ is equal to $|G|$ for every $i\geq 0$. The following theorem says that when $M$ is finitely generated this  situation cannot happen and that the exponent of $H^i (G, M)$ eventually becomes small at high dimensions.  

\begin{thm}[Pakianathan \cite{pakianathan}]\label{jonathan's lemma}
Let $G=(\bbZ/p)^r$ and M be a finitely generated $\bbZ G$-module. Then, there is an integer $N$ such that the exponent of $H^i(G,M)$ divides $p$ for all $i\geq N$.
\end{thm}

\begin{proof}
By Theorem 7.4.1 in \cite[p.~87]{evens}, $H^*(G,M)$ is a finitely generated module over the ring $H^*(G,\bbZ)$. Let $m_1,...,m_k$ be homogeneous elements generating $H^*(G,M)$ as an $H^*(G,\bbZ)$-module and let $N= 1+\max _j\{\deg m_j\}$.   If $x \in H^i(G,M)$ such that $i\geq N$, then we can write $x=\Sigma_{j=1}^k\alpha_j m_j$ for some homogeneous elements $\alpha_j$ in $H^* (G, \bbZ)$ with $\deg \alpha _j \geq 1$ for all $j$. Since $exp H^i (G, \bbZ)$ divides $p$ for all $i\geq 1$, we have $p\alpha_j=0$ for all $j$. Hence we obtain $px=\Sigma_{j=1}^kp\alpha_jm_j=0$ as desired.
\end{proof}

\section{Proof of the Main Theorem}
\label{sect:maintheorem}

Let $G=(\bbZ/p)^r$ and $k$, $l$ be positive integers. We will show that there is an integer $N$ such that if $G$ acts freely and cellularly on a CW-complex $X$ homotopy equivalent to $S^{n_1} \times \dots \times S^{n_k}$ where $|n_i-n_j|\leq l$ and $n_i \geq N$ for all $i,j$, then $r \leq k$.

Suppose that $G$ acts freely and cellularly on some CW-complex $X$ homotopy equivalent to $S^{n_1} \times \dots \times S^{n_k}$ where $|n_i -n_j|\leq l$ for all $i,j$. Let $n=\max\{n_i:i=1,\dots,k\}$ and let $a_i=n-n_i$ for all $i$. Consider the cellular chain complex $C_*(X)$ of the CW-complex $X$. The complex $C_*(X)$ is a nonnegative, connected, and finite-dimensional chain complex of free $\bbZ G$-modules and has nonzero homology at the following dimensions other than dimension zero: 
\[
\begin{matrix}
(1) & n-a_1, n-a_2,\dots,n-a_k \\
(2) & 2n-a_1-a_2, 2n-a_1-a_3,\dots,2n-a_{k-1}-a_k \\
& \vdots \\
(j) & jn-(a_1+\dots+a_j), \dots , jn-(a_{k-j+1}+\dots+a_k) \\
&\vdots \\
(k) & kn-(a_1+a_2+\dots+a_k). 
\end{matrix}
\]

If $n>lk$, then we have $n> a_1+\cdots +a_k$ which implies that for all $j$, the dimensions listed on the $j$-th row are strictly larger than the dimensions listed on the $(j-1)$-st row. Since this fact is crucial for our argument,  we will assume that the integer $N$ in the statement of the theorem satisfies $N>lk$ to guarantee that this condition holds. 

Now we can apply Habegger's argument given in Theorem \ref{gluing lemma} to glue all the homology groups at the dimensions listed on the $j$-th row above to the homology at dimension $jn$ for all $j=1,\dots, k$. The resulting complex $D_*$ is a connected, finite-dimensional chain complex of free $\bbZ G$-modules which has homology only at dimensions $0,n,2n,\dots,kn$. Let $M_j:=H_{jn} (D_*)$ for all $j=1,\dots,k$. Note that by construction $M_j$ is a finitely generated $\bbZ G$-module for all $j$ since syzygies of finitely generated $\bbZ G$-modules are finitely generated when $G$ is a finite group.

Now we can apply Theorem \ref{jonathan's lemma} to find an integer $N_j $ for each $j$  such that if $i\geq N_j$, then $exp H^i (G, M_j)$ divides $p$. Suppose that  for a fixed $G=(\bbZ /p)^r$, $k$, and $l$, there are only finitely many possibilities for $\bbZ G$-modules $M_j$'s up to stable equivalence. Then by taking the maximum of $N_j$'s over all possible $M_j$'s, we can find an integer $N^{\max}_j $ for each $j$  such that if $i\geq N^{\max}_j$, then $exp H^i (G, M_j)$ divides $p$ for all possible $M_j$'s that may occur. Then we can take $N=\max _j N_j^{\max}$ and complete the proof in the following way. By Theorem \ref{thm:browder}, we have $ |G|=p^r$ divides $$\prod _{j=1} ^k H^{jn+1} (G, H_{jn} (D_*))=\prod _{j=1} ^k H^{jn+1} (G, M_j).$$ So, if $n\geq N$, then $p^r$ divides $p^k$ which gives $r\leq k$ as desired.

Hence to complete the proof, it only remains to show that for fixed $G=(\bbZ /p)^r$, $k$, and $l$, there are only finitely many possibilities for $\bbZ G$-modules $M_j$'s up to stable equivalence. To show this, first note that for a fixed $l$, there are finitely many $k$-tuples $(a_1,...,a_k)$ with the property that $0 \leq a_i \leq l$ for all $i$. So we can assume that we have a fixed $k$-tuple $(a_1,\dots, a_k)$. Let us also fix an integer $j$ and show there are only finitely many possibilities for $M_j=H_{jn} (D_*)$.

Let $s_1< \dots <  s_m$ be a sequence of integers such that $\{ n-s_1,\dots, n-s_m\}$ is the set of all distinct dimensions on the $j$-th row of the above diagram. Then, by Theorem \ref{gluing lemma} the module $M_j$ has a filtration
$$ 0=K_0 \subseteq K_1 \subseteq \cdots \subseteq K_m=M_j$$
such that $K_i /K_{i-1} \cong \Omega ^{s_i} (A_i)$ where $A_i=H_{jn-s_i} (X)$.  For all $i$, the module $A_i$ is a $\bbZ$-free $\bbZ G $-module with $\bbZ$-rank less than or equal to ${k \choose j}$, so by Jordan-Zassenhaus theorem (see Corollary (79.12) in \cite[p.~563]{curtis}), there are only finitely many possibilities for $A_i$'s up to isomorphism. 

We will inductively show that there exist only finitely many possibilities for $K_i$'s up to stable equivalence. For $i=1$, we have $K_1=\Omega^{s_1} (A_1)$ so this follows from the fact that there are only finitely many possibilities for $A_1$ and that syzygies are well-defined up to stable equivalence. For $i>1$, consider the following short exact sequence:
\[\minCDarrowwidth20pt\begin{CD} 0 @>>> K_{i-1} @>>> K_i @>>> \Omega^{s_i}A_i @>>> 0. \end{CD}\]
By induction we know that there are only a finite number of possibilities for $K_{i-1}$'s up to stable equivalence. By a similar argument as above, the same is true for $\Omega^{s_i} (A_i)$. The extensions like the ones above are classified by the ext-group $\Ext^1 _{\bbZ G} (\Omega ^{s_i} (A_i), K_{i-1} )$ and since both modules are $\bbZ $-free, these ext-groups are well-defined up to stable equivalence. So, it remains to show that $$\Ext^1_{\bbZ G}(\Omega^{s_i} (A_i),K_{i-1})=\Ext^{s_i+1}_{\bbZ G}(A_i,K_{i-1})$$ is a finite group. Note that since both $A_i$ and $K_{i-1}$ are finitely generated, $\Ext^{s_i+1}_{\bbZ G}(A_i,K_{i-1})$ is a finitely generated abelian group. Moreover, since $A_i$ is $\bbZ$-free, it has an exponent divisible by $|G|$. So, $\Ext^{s_i+1}_{\bbZ G}(A_i,K_{i-1})$ is a finite group.  This completes the proof of Theorem \ref{thm:maintheorem}.

We conclude this section with a generalization of Theorem \ref{thm:maintheorem} to non-free actions. The exact statement is as follows.

\begin{thm}\label{thm:genmaintheorem} Let $G=(\bbZ/p)^r$ and $k,l$ be positive integers. Then there exists an integer $N$ (depending on $k,l$ and the group $G$) such that if $G$ acts cellularly on a finite dimensional CW-complex $X$ homotopy equivalent to $S^{n_1} \times \dots \times S^{n_k}$ where $n_i \geq N$ and $|n_i-n_j|\leq l$ for all $i,j$, then $r-s\leq k$ where $s$ is the largest integer such that $|G_x|=p^s$ for some $x\in X$.
\end{thm}

\begin{proof} Let $C_* (X)$ denote the cellular chain complex of $X$ and let $\varepsilon : C_* (X) \to \bbZ$ be the map induced by the constant map $X \to pt$. The arguments in the proof of Theorem \ref{thm:maintheorem} can be repeated to prove that there is an integer $N$ such that if $n_i \geq N$ and $|n_i-n_j |\leq l$ for all $i,j$, then $$p^k \tate ^0 (G, \bbZ ) \subseteq \image \{ \varepsilon : \tate ^0 (G, C_*(X) )\to \tate ^0 (G, \bbZ ) \}.$$ 
This can be seen easily by a fourth quadrant spectral sequence argument or by the filtration argument given in the proof of Theorem \ref{thm:browder}. So, using the inclusion above, we obtain that $|G|=p^r$ divides $p^k \cdot exp \tate ^0 (G, C_*(X))$. Hence the proof will be complete if we can show that $exp \tate ^0 (G, C_*(X) )$ divides $p^s$ where $s$ is the largest integer such that $|G_x|=p^s$ for some $x\in X$.

To prove this last statement we use a theorem due to J.~F.~Carlson which gives a bound for the exponent of $\tate ^* (G, M)$ in terms of the complexity of $M$ (see \cite[Theorem 2.6]{carlson}). In fact we will be using a 
generalization of this result to hypercohomology due to A. Adem (see \cite[Theorem 3.1 and 3.2]{adem}). Let $M=\oplus _{H\in {\rm Iso}(X)} \bbZ [G/H]$ where ${\rm Iso} (X)$ denotes the set of isotropy subgroups of the $G$-action on $X$. Let $s$ be the largest integer such that $|H|=p^s$ for some $H \in {\rm Iso} (X)$. Then, we can find homogeneous elements $\zeta _1,\dots , \zeta _s \in H^* (G, \bbZ)$ of positive degree which cover the variety of $M$ (see \cite{carlson} for more details). This implies, in particular, that $L_{\zeta _1} \otimes \cdots \otimes L_{\zeta _s } \otimes C_* (X)$ is a finite length chain complex of projective $\bbZ G$-modules. Repeating the argument in the proof of Theorem 3.1 in \cite{adem}, one easily gets $\exp \tate ^* (G, C_* (X) ) $ divides $\prod _{i=1} ^s \exp \zeta _i.$ Since $H^i (G, \bbZ )$ has exponent 1 or $p$ for all $i \geq 1$, we get $exp \tate ^0 (G, C_* (X))$ divides $p^s$ as desired.   
\end{proof}

\bigskip

\noindent
Department of Mathematics\\
Bilkent University,\\
Ankara, 06800, Turkey. \\
E-mail addresses: okutan@fen.bilkent.edu.tr, yalcine@fen.bilkent.edu.tr \\

\end{document}